\documentclass{amsart}

\usepackage{amsmath,amssymb,amsthm}

\usepackage{graphicx}
\usepackage{epsfig,subfig}

\theoremstyle{plain}
\newtheorem{theorem}{Theorem}[section]
\newtheorem{corollary}[theorem]{Corollary}
\newtheorem{proposition}[theorem]{Proposition}
\newtheorem{lemma}[theorem]{Lemma}
\newtheorem{remark}{Remark}[section]

\newcommand{\mdef}[1]{\textsl{#1}}
\newcommand{\norm}[1]{\|#1\|}
\newcommand{\indnorm}[1]{|\hspace{-.3ex}|\hspace{-.3ex}|#1|\hspace{-.3ex}|\hspace{-.3ex}|}

\newcommand{\abs}[1]{|#1|}
\newcommand{\posp}[1]{(#1)^+}

\newcommand{\bposp}[1]{\left(#1\right)^+}
\newcommand{\set}[1]{\{#1\}}
\newcommand{\ip}[2]{\sum_k #1_k}
\newcommand{\R}{\mathbb{R}}
\newcommand{\Rnp}{\R^{n}_{\ge0}}
\newcommand{\Rnnp}{\R^{n\times n}_{\ge0}}

\newcommand{\C}{\mathbb{C}}
\newcommand{\e}[1]{{e}^{#1}}
\newcommand{\vun}{\mathbf{1}}
\newcommand{\Gr}[1]{G(#1)}
\newcommand{\SpArb}[1]{\mathcal{A}(#1)}
\newcommand{\tauB}{\tau_B}
\renewcommand{\vec}[1]{\mathbf{#1}}
\newcommand{\Lip}[1]{L_{\max}(#1)}
\newcommand{\Lipm}[1]{L_{\min}(#1)}

\newcommand{\eps}{\varepsilon}
\DeclareMathOperator{\ind}{ind}
\DeclareMathOperator{\diag}{diag}
\DeclareMathOperator{\rank}{rank}

\title[Self-validating web rankings]{Multiple equilibria of nonhomogeneous Markov chains and self-validating web rankings}
\thanks{Part of the results of this paper where announced in the Proceedings of the 2nd Multidisciplinary International Symposium on Positive Systems: Theory and Applications (POSTA 06)~\cite{AGN06}.}

\author[M. Akian]{Marianne Akian}
\address{INRIA, Domaine de Voluceau, BP105, 78153~Le~Chesnay~C\'edex, France}
\email[Marianne Akian]{Marianne.Akian@inria.fr}
\author[S. Gaubert]{St\'ephane Gaubert}
\address{INRIA, Domaine de Voluceau, BP105, 78153~Le~Chesnay~C\'edex, France}
\email[St\'ephane Gaubert]{Stephane.Gaubert@inria.fr}

\author[L. Ninove]{Laure Ninove}
\address{Department of Mathematical Engineering, Universit\'e catholique de Louvain, 1348~Louvain-la-Neuve, Belgium}
\email[Laure Ninove]{Laure.Ninove@uclouvain.be}

\thanks{The first two authors were partially supported by the joint RFBR-CNRS grant number 05-01-02807. Part of this work was done while the third author was visiting INRIA. 
This author was supported by a FNRS research fellowship and partially by research programs IAP~VI/4 DYSCO (Belgian Federal Science Policy) and ARC Large Graphs and Networks (Communaut\'e Fran\c{c}aise de Belgique). The scientific responsibility rests with its authors.}

\date{September 3, 2007}

\keywords{PageRank, stochastic matrix, Markov chain, stationary vector, iterative process, projective metric, Perron--Frobenius theory}

\subjclass[2000]{15A51, 47H07, 15A48, 15A18, 47H12, 60J10, 68P20}

\begin{document}
\maketitle

\begin{abstract}
PageRank is a ranking of the webpages used by the search engine Google to determine in which order to display webpages.  It measures how often a given webpage would be visited by a random surfer on the webgraph.  It seems realistic that this ranking may influence the reputation of the webpages, and therefore the behavior of websurfers.  We propose a simple model taking in account the mutual influence between webranking and websurfing.

We study this model by considering the following iteration on the set of stochastic row vectors: $\vec r(s+1)=\vec u_T(\vec r(s))$, where for all vector $\vec x$, $\vec u_T(\vec x)$ is the unique invariant measure of a primitive stochastic matrix $M_T(\vec x)$ constructed from the adjacency matrix of the webgraph.  The parameter $T>0$ is fixed and measures the confidence of the surfer in the ranking.
We call $T$-PageRank the limit, if it exists, of these iterates when $s$ tends to infinity.
We also consider the simple iteration defined by $ \vec{\tilde r}(s+1)= \vec{\tilde r}(s)M_T(\vec{\tilde r}(s))$.

We prove that, when $T$ is large enough, the fixed point of these iterations is unique and the convergence is global on the domain. But for small values of $T$, at least when the matrix $C$ is positive, there are always several fixed points.

Our analysis uses results of nonlinear Perron--Frobenius theory, Hilbert projective metric and Birkhoff's coefficient of ergodicity.
\end{abstract}

\section{Introduction}

\paragraph{Motivation}

The PageRank algorithm~\cite{BP98} is at the heart of one of the most popular Web search engines.
Its basic idea is to use the graph structure of the Web in order to assign to each webpage a score.
Basically, the PageRank score attributed to the webpages measures \emph{how often a given page would be visited by a random walker on the webgraph}.
Formally, let $C=[C_{ij}]$ be the $n\times n$ adjacency matrix of the webgraph, so that $C_{ij}=1$ if there is a hyperlink from page $i$ to page $j$ and $C_{ij}=0$ otherwise.  For simplicity, we first assume that $C$ is irreducible, that is the webgraph is strongly connected.
Imagine that, when visiting a page $i$, a websurfer chooses randomly the next webpage he/she will visit, among the pages referenced by page $i$, with the uniform distribution.
The trajectory of such a websurfer is a Markov chain with transition matrix $M=[M_{ij}]$, given by
\[
  M_{ij}=\frac{C_{ij}}{\sum_kC_{ik}}\enspace.
\]
In its most basic version, the PageRank vector $\vec r$, the entries of which give the PageRank score of the webpages, is simply defined as the stationary distribution of this random walk on the webgraph.  Thus $\vec r$ is the invariant measure of the irreducible matrix $M$, that is the unique stochastic vector such that
\[
  \vec r = \vec r M\enspace.
\]

However, the assumption that a websurfer makes uniform draws on the web\-graph may seem unrealistic: a websurfer could have an a priori idea of the value of webpages, therefore favoring pages from reputed sites.  The webrank may influence the reputation of the websites, and hence it may influence the behavior of the websurfers, which ultimately may influence the webrank.
In this paper, we propose a simple model taking into account the mutual influence between webranking and websurfing.

\smallskip
\paragraph{The $T$-PageRank}
We consider a sequence of stochastic vectors, representing successive webranks, $\vec r(0), \vec r(1),\dots$, which is defined inductively as follows.  The current ranking $\vec r(s)$ induces a random walk on the webgraph.  We assume that the websurfer moves from page $i$ to page $j$ with probability proportional to $C_{ij}\e{E(\vec r(s)_j)/T}$, where $E$ is an increasing function, the \mdef{energy}, and $T>0$ is a fixed positive parameter, the \mdef{temperature}.  The websurfer's trajectory is therefore a Markov chain with transition matrix $M(\vec r(s))$, where $M(\vec x)$ is defined for all vector $\vec x$ by
\[
  M(\vec x)_{ij}=\frac{C_{ij}\e{E(\vec x_j)/T}}{\sum_kC_{ik}\e{E(\vec x_k)/T}}\enspace.
\]
The unique stationary distribution of this Markov chain, i.e.\ the invariant measure of the matrix $M(\vec r(s))$, is then used to update the webrank.  Thus,
\begin{subequations}\label{eq:iter-uT}
\begin{equation}\label{eq:iter-uT-a}
  \vec r(s+1)=\vec u(\vec r(s))\enspace,
\end{equation}
where, for all vector $\vec x$, $\vec u(\vec x)$ is the unique stochastic vector such that
\begin{equation}\label{eq:iter-uT-b}
  \vec u(\vec x)=\vec u(\vec x)M(\vec x)\enspace.
\end{equation}
\end{subequations}
We call \mdef{$T$-PageRank} the limit of $\vec r(s)$ when $s$ tends to infinity, if it exists, or a fixed point of the map $\vec u$.

Note that if $\vec r(0)$ is the uniform distribution, then $\vec r(1)$ is the classical PageRank.
Note also that the temperature $T$ measures the randomness of the process.  If $T$ is small, with overwhelming probability, the websurfer shall move from page $i$ to one of the pages $j$ referenced by page $i$ of best rank, i.e.\ maximizing $\vec r(s)_j$, whereas if $T=\infty$, the websurfer shall draw the next page among the pages $j$ referenced by the page $i$, with the uniform distribution, as in the standard webrank definition.  For $T=\infty$, the $T$-PageRank coincides with the classical PageRank, because in this case $M(\vec r(s))=M$.

We also consider the simple iteration defined by
\begin{equation}\label{eq:iter-fT}
  \vec{\tilde r}(s+1)= \vec f(\vec{\tilde r}(s))\enspace, \quad \text{ with }
  \vec f(\vec x)=\vec x M(\vec x)\enspace,
\end{equation}
where $\vec{\tilde r}(0)$ is an arbitrary stochastic vector.  From a computational point of view, this is similar to the standard power method.

\smallskip
\paragraph{Main results}
Our first main result shows that, if the temperature $T$ is sufficiently large, the $T$-PageRank exists, is unique and does not depend on the initial ranking.  Moreover, at least when the matrix $C$ is primitive, the generalized power algorithm~$(\ref{eq:iter-fT})$ can be used to compute the $T$-PageRank.

\begin{theorem}\label{thm-intro:T-large}
  If $T\ge n\,\mathrm{Lip}(E)$, where $\mathrm{Lip}(E)$ is the Lipschitz constant of the function $E$, then the map $\vec u$ given by~$(\ref{eq:iter-uT-b})$ has a unique fixed point and the iterates~$(\ref{eq:iter-uT-a})$ converge to it for every initial ranking.
  Moreover, if $C$ is primitive and if $T$ is large enough, the iterates~$(\ref{eq:iter-fT})$ converge to this unique fixed point for every initial ranking.
\end{theorem}

On the other hand, for small values of $T$, several $T$-PageRanks exist, depending on the choice of the initial ranking.
In some cases, the $T$-PageRank does nothing but validating the initial ``belief'' in the interest of pages given by the initial ranking.
Consider for instance the graph given in Figure~\ref{fig:ex-self-validating} with adjacency matrix $C=\left(\begin{smallmatrix}0&1&1\\1&1&0\\1&0&1\end{smallmatrix}\right)$, and let $E(x)=x$ for all $x$ and $T=\frac{1}{4}$.  Let $\vec r(0)=\vec{\tilde r}(0)=\left(\frac{1}{3}\,\,\,\frac{1}{3}+\eps\,\,\,\frac{1}{3}-\eps \right)$ for an arbitrary small $\eps>0$.  Then the iterates~$(\ref{eq:iter-uT})$ and~$(\ref{eq:iter-fT})$ converge to a $T$-PageRank close to $\left(0.021\,\,\,0.978\,\,\,0.001 \right)$, so the initial belief that the node~$2$ is more interesting than node~$3$ has strongly increased.
\begin{figure}[htb]
\begin{center}
    \includegraphics[width=5.5cm]{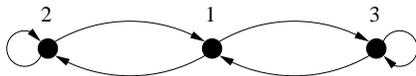}
    \caption{For this graph, self-validating effects appear for small temperatures.}
    \label{fig:ex-self-validating}
\end{center}
\end{figure}
Our second main result shows that the existence of multiple $T$-PageRanks is in fact a general feature, when $T$ is small enough.
\begin{theorem}\label{thm-intro:T-small}
  If $C$ has at least two positive diagonal entries, then multiple $T$-PageRanks exist for $T$ small enough.
\end{theorem}

These two theorems follow respectively from Theorems~\ref{thm:u_tau-fxpt-conv} and~\ref{thm:f_tau-fxpt-conv}, and from Theorem~\ref{thm:mult-pt-fx}, which are stated in a more general framework.

We also study variants, bringing more realistic models of the behavior of the websurfer.  This includes the presence of a ``damping factor'', as in the standard definition of Google's PageRank. We also consider the situation where the websurfer may take the webrank into account only when visiting a special page (the search engine's webpage).  Similar conclusions apply to such variants.

\smallskip
\paragraph{Method}  In order to analyze the map~$(\ref{eq:iter-uT-b})$, we first use Tutte's Matrix Tree Theorem~\cite{Tut84} to express explicitly the invariant measure $\vec u(\vec x)$ in term of the entries of $M(\vec x)$.  Then we study the convergence and the fixed points of~$(\ref{eq:iter-uT})$ by using results of nonlinear Perron--Frobenius theory due to Nussbaum~\cite{Nus88} and Krause~\cite{Kra86}.
To analyze the iteration~$(\ref{eq:iter-fT})$, we use two different approaches, which give convergence results under distinct technical assumptions.  Our first approach is to show that, if $T$ is sufficiently large, the map $\vec f$ from the simplex to itself is a contraction for some norm.  Our second approach uses Hilbert's projective metric and Birkhoff's coefficient of ergodicity.

\smallskip
\paragraph{Related work} 
Several variants of the PageRank are considered in the literature in order to have a more realistic model of the behavior of the websurfers.  For instance, some authors propose to introduce the browser's back button in the model~\cite{FKKRRRST01,BM05,Syd04}.  Other try to develop a topic-sensitive ranking by considering a surfer whose trajectory on the webgraph depends on his query or bookmarks~\cite{Hav02,JW03,RD02}. 
Moreover, the standard PageRank vector is simply the dominant left-eigenvector of some stochastic matrix.  It has therefore received attention from the linear algebra community.  Current research concerns for instance efficient computation methods~\cite{IK06,KHG04,LM06book,LM06,SC05}, or sensitivity analysis~\cite{Kir06,LM04,LM06book} of the Page\-Rank.

Our use of transition probabilities proportional to $C_{ij}\e{E(\vec r(s)_j)/T}$, where $E$ is an energy function and $T$ the temperature, is reminiscent of simulated annealing algorithms.  For a reference, see Catoni~\cite{Cat99}.  In the context of opinion formation, Holyst et al.~\cite{HKS01} study a social impact model where the probability that an individual changes his opinion depends on a ``social temperature'' $T$, which measures the randomness of the process.

The iteration~$(\ref{eq:iter-uT})$
can be studied in the settings of nonlinear Perron--Frobenius theory.  Many works exist in the literature in order to generalize the classical Perron--Frobenius theorems to nonlinear maps on cone satisfying some hypotheses as, for instance, primitivity, positivity or homogeneity.  See Nussbaum~\cite{Nus88} and also~\cite{GW04,Met05} for recent references.

The iteration~$(\ref{eq:iter-fT})$ has been studied by several authors in an abstract setting.
Artzrouni and Gavard~\cite{AG00} analyze their dynamics when $\vec x(s)$ behaves asymptotically like $\lambda^s \vec x_*$ for some $\lambda\neq1$ and $\vec x_*$.  For $\lambda=1$, it can be useful to look at the stability of a linearization of the system near one of its fixed points~\cite{Cas89}.  When $M(\vec x)$ is stochastic and satisfies certain monotonicity conditions, Conlisk~\cite{Con92} proves the convergence of the iterates~$(\ref{eq:iter-fT})$ to a stable limit.  Lorenz~\cite{Lor05} proves their convergence for \emph{column-}stochastic matrices satisfying classic properties of opinion dynamics models.  In~\cite{Lor06}, he experimentally studies a reformulation of these models with stochastic matrices $M(\vec x)$, where $\vec x$ is an opinion distribution vector.
Iterations like~$(\ref{eq:iter-fT})$ could also be studied in the setting of nonhomogeneous products of matrices.  In this case, iterations like $\vec x(s+1)=\vec x(s)M(s)$ are considered, where the matrices do not depend explicitly on $\vec x$.  Two classical approaches to study their dynamics and convergence are the use of ergodicity coefficients~\cite{Art96,Har02,Sen81,Wol63} or of the joint spectral radius~\cite{Gur95,Har02,JLM03}.
However, the main results of this paper can not be deduced from these works.

\smallskip
\paragraph{Outline of the paper} 
We first introduce some preliminaries in Section~\ref{sec:preliminaries}.  In Section~\ref{sec:main-section}, we prove the main results of this paper, about the existence, uniqueness or multiplicity of the $T$-PageRank, and the convergence of the iterates~$(\ref{eq:iter-uT})$ and~$(\ref{eq:iter-fT})$.  Then, in Section~\ref{sec:refinement}, we introduce a refinement of our model, inspired by the damping factor of the classical PageRank algorithm.  This variant has the advantage of allowing one to work with non strongly connected webgraphs.  Moreover, it gives a more realistic model of the behavior of a standard websurfer.  Section~\ref{sec:particular-cases} is devoted to the estimation of the \mdef{critical temperature}, that is the temperature corresponding to the loss of the uniqueness of the $T$-PageRank, for some particular cases.  We shall see that, even for very small or regular webgraphs, the $T$-PageRank can have a complex behavior.  Finally, in Section~\ref{sec:application-web}, we experiment the $T$-PageRank algorithm on a large-scale example.

\section{Preliminaries}\label{sec:preliminaries}

In this section, we introduce some notations and recall some classical concepts about nonnegative matrices and graphs, and about iterated maps on cones (for background, see~\cite{BP94,HoJo,Nus88,Sen81}).

We denote by $\R_{\ge0}$ and $\R_{>0}$ the sets of nonnegative and positive numbers respectively.
The \mdef{simplex} $\Sigma=\set{\vec x\in\R^n_{\ge0}\colon \sum_i\vec x_i=1}$ is the set of \mdef{stochastic row vectors}.  Its relative interior is denoted by $\mathrm{int}(\Sigma)=\Sigma\cap\R^n_{>0}$, and its boundary by $\partial\Sigma=\Sigma\backslash\mathrm{int}(\Sigma)$.
We also write $\vec x\ge\vec y$ if the vector $\vec x-\vec y\in\Rnp$ and $\vec x>\vec y$ if the vector $\vec x-\vec y\in\R^n_{>0}$.
For some vector $\vec d$, let $D=\diag(\vec d)$ be the diagonal matrix such that $D_{ii}=\vec d_i$ for all $i$ and $D_{ij}=0$ for all $i\neq j$.
Finally, let $\vun$ denote the column vector of all ones, and $\delta_{ij}$ denote the Kronecker delta.

Let us begin with nonnegative matrices and graphs~\cite{BP94,HoJo,Sen81}. Let $A\in\Rnnp$ be a nonnegative matrix.  Its associated \mdef{directed graph} $\Gr{A}$ is a graph with nodes $1,2,\dots,n$, and a directed edge $(i,j)$ if and only if $A_{ij}>0$.  
This graph is \mdef{strongly connected}, that is there exists a directed path between each pair of nodes, if and only if $A$ is \mdef{irreducible}.  It is moreover \mdef{aperiodic} if and only if $A$ is \mdef{primitive}.
Let $r$ be a node of $\Gr{A}$. A directed subgraph $R$ of $\Gr{A}$ which contains no directed cycles and such that, for each node $i\neq r$, there is exactly one edge leaving $i$ in $R$, is called a \mdef{spanning arborescence} rooted at $r$. The set of spanning arborescences of $\Gr{A}$ rooted at $r$ is denoted by $\SpArb{r}$.

The transition probabilities of a \mdef{finite Markov chain}, or equivalently a \mdef{random walk} on a directed graph with weighted edges, can be represented by a \mdef{stochastic matrix}, that is a matrix $A\in\Rnnp$ such that $A\vun=\vun$.  A stochastic vector $\vec u\in\Sigma$ is an \mdef{invariant measure} of $A$ if $\vec uA=\vec u$.  The Perron--Frobenius theory ensures that an irreducible stochastic matrix has a unique invariant measure, which is moreover positive.  An expression of this invariant measure is given by Tutte's Matrix Tree Theorem~\cite{Tut84} 
(for a proof, see Section~$9.6$ of~\cite{BR91}).
\begin{theorem}[Matrix Tree Theorem, Tutte]\label{thm:Tutte}
    Let $A\in\Rnnp$ be an irreducible stochastic matrix,
    and let $\vec u$ be its invariant measure.  Then
    $\vec u={\vec v}/{\sum_i{\vec v_i}}$, where for all index $r$
    \begin{equation}\label{eq:mtt}
        \vec v_r = \sum_{R\in \SpArb{r}}\prod_{(i,j)\in R}A_{ij}\enspace.
    \end{equation}
\end{theorem}

We now present some useful concepts about iterated maps on cones~\cite{Nus88,Sen81}.  A function $\vec h$ is \mdef{subhomogeneous} on a set $U$ if $\lambda \vec h(\vec x)\le\vec h(\lambda \vec x)$ for all $\vec x\in U$ and every $0\le\lambda\le1$. It is \mdef{order-preserving} on $U$ if $\vec h(\vec x)\le\vec h(\vec y)$ for all $\vec x,\vec y\in U$ such that $\vec x\le\vec y$.
\mdef{Hilbert's projective metric} $d_H$, is defined as
\[
    d_H\colon\R^n_{>0}\times\R^n_{>0}\to\R_{\ge0}
    \colon (\vec x,\vec y)\mapsto \max_{i,j}\ln\frac{\vec x_i\vec y_j}{\vec y_i\vec x_j}\enspace.
\]
This metric defines a distance on $\mathrm{int}(\Sigma)$.  An \mdef{induced projective metric} between two positive matrices can be defined as
\[
    d_H\colon\R^{n\times n}_{>0}\times\R^{n\times n}_{>0}\to\R_{\ge0}\colon
    (A,B)\mapsto \sup_{\vec x\in\R^n_{>0}}\, d_H(\vec xA,\vec xB)\enspace.
\]
The \mdef{coefficient of ergodicity} $\tauB$, also known as Birkhoff's contraction coefficient, is defined for a nonnegative matrix $A$ having no zero column as
\[
    \tauB(A)=\sup_{\substack{\vec x,\vec y\in\R^n_{>0}\\\vec x\neq\lambda \vec y}}
    \, \frac{d_H(\vec xA,\vec yA)}{d_H(\vec x,\vec y)}\enspace.
\]
This coefficient $\tauB(A)\in{[0,1]}$, and $\tauB(A)<1$ if and only if $A$ is positive.

The following theorem deals with the fixed points and the convergence of iterated maps on a cone.
This theorem is a simple formulation of the very general Theorems~$2.5$ and~$2.7$ (or more precisely of Corollaries~$2.2$ and~$2.5$) of Nussbaum in~\cite{Nus88}.
\begin{theorem}[Nussbaum]\label{thm:Nussbaum}
    Let $\vec h\colon\R^n_{>0}\to\R^n_{>0}$ be a continuous, order-preserving map
    which  is subhomogeneous on $\mathrm{int}(\Sigma)$.
    Suppose moreover that if $\vec x\in\mathrm{int}(\Sigma)$ is an eigenvector
    of $\vec h$, then $\vec h$ is continuously differentiable on an open neighborhood of
    $\vec x$ and the matrix $\vec h'(\vec x)$ is nonnegative and irreducible.
    Then $\vec h$ has at most one eigenvector in $\mathrm{int}(\Sigma)$.  As a
    consequence, the map $\vec u\colon\mathrm{int}(\Sigma)\to\mathrm{int}(\Sigma)\colon
    \vec x\mapsto{\vec h(\vec x)}/{\sum_i{\vec h(\vec x)_i}}$ has at most one fixed point.

    \noindent Moreover, if $\vec h$ has an eigenvector $\vec x\in\mathrm{int}(\Sigma)$, and if
    $\vec h'(\vec x)$ is primitive, then all the orbits of $\vec u$ converge to
    its unique fixed point $\vec x$.
\end{theorem}

We end with this preliminary section with a result of Krause~\cite{Kra86}, also dealing with the fixed point and the convergence of iterated maps on the positive orthant.
\begin{theorem}[Krause]\label{thm:Krause}
    Let $p\colon\R^n_{\ge0}\to\R_{\ge0}$ be a continuous map which is not identically $0$ and 
    such that $p(\lambda\vec x)=\lambda p(\vec x)$ for all $\vec x\in\R^n_{\ge0}$, $\lambda\ge0$;
    and $p(\vec x)\le p(\vec y)$ for all $0\le\vec x\le \vec y$.
    Let $U=\set{\vec x\in\R^n_{\ge0}\colon p(\vec x)=1}$.

    \noindent Let $\vec h\colon\R^n_{\ge0}\to\R^n_{>0}$ be a map such that there exist $\alpha,\beta>0$ and $\vec v\in\R^n_{>0}$ 
    such that $\alpha\vec v\le\vec h(\vec x)\le\beta\vec v$ for all $\vec x\in U$.
    Suppose also that
    $\lambda\vec h(\vec x)\le\vec h(\vec y)$ for all $\vec x,\vec y\in U$ and $0\le\lambda\le1$ such that $\lambda\vec x\le\vec y$;
    and that $\lambda\vec h(\vec x)<\vec h(\vec y)$ if moreover $\lambda<1$ and $\lambda\vec x\neq\vec y$.
    Then $\vec h$ has a unique eigenvector in $U$.  As a
    consequence, the map $\vec u\colon U\to\mathrm{int}(U)\colon
    \vec x\mapsto{\vec h(\vec x)}/{p(\vec h(\vec x))}$ has a unique fixed point in $U$.
    Moreover, all the orbits of $\vec u$ converge to this unique fixed point.
\end{theorem}


\section{Existence, uniqueness, and approximation of the $T$-PageRank}\label{sec:main-section}

\subsection{Hypotheses}

Let $C$ be an $n\times n$ irreducible nonnegative matrix.
For all \mdef{temperature} $T>0$, and all $\vec x\in\Sigma$, let $M_T(\vec x)$ be the irreducible stochastic matrix such that
\[
    M_T(\vec x)_{ij} 
    = \frac{C_{ij}\,g_T(\vec x_j)}{\sum_k C_{ik}\,g_T(\vec x_k)}\enspace,
\]
where $g_T\colon {[0,1]}\to\R_{>0}$ is a continuously differentiable map with $g_T(0)=1$.  We suppose moreover that $g_T$ is increasing with $g_T'\colon {[0,1]}\to\R_{>0}$ and make the following assumptions on the asymptotic behavior of $g_T$
\begin{gather}
  \lim_{T\to0}\Lipm{g_T}=\infty\enspace,  \tag{${A}_0$}\label{eq:hyp-g_T-0}\\
  \lim_{T\to\infty}\Lip{g_T}=0\enspace,  \tag{${A}_\infty$}\label{eq:hyp-g_T-infty}
\end{gather}
where $\Lipm{g_T}$ and $\Lip{g_T}$ are defined as
\[
  \Lipm{g_T}=\min_{x\in{[0,1]}}\frac{g_T'(x)}{g_T(x)} \quad \text{ and } \quad
  \Lip{g_T}=\max_{x\in{[0,1]}}\frac{g_T'(x)}{g_T(x)}\enspace.
\]

\paragraph{Example}
  These hypotheses are satisfied in particular for $g(x)=\e{E(x)/T}$, 
  where the energy function $E$ is continuously differentiable, 
  with $E(0)=0$ and $E'(x)>0$ for all $x\in{[0,1]}$. 
  In this case, $\Lipm{g_T}=\min_{x\in{[0,1]}}{E'(x)}/{T}$ and $\Lip{g_T}=\max_{x\in{[0,1]}}{E'(x)}/{T}$.


\subsection{Preliminary results}

The following elementary lemmas will be useful in the sequel.
\begin{lemma}\label{lem:bornes-ln-Lmax-Lmin}
  For all $x,y\in{[0,1]}$,
\[
    \Lipm{g_T}\,\posp{x-y}
    \le\bposp{\ln \frac{g_T(x)}{g_T(y)}}
    \le\Lip{g_T}\,\posp{x-y}\enspace,
  \]
  where for all $x\in\R$, $x^+=\max\set{0,x}$.
  Moreover, if $x,y\neq0$,
  \[
    \bposp{\ln \frac{g_T(x)}{g_T(y)}}
    \le\Lip{g_T}\,\bposp{\ln \frac{x}{y}}\enspace.
  \]
\end{lemma}
\begin{proof}
  Let $x,y\in{[0,1]}$.  Then
  \begin{align*}
    \posp{\ln g_T(x)-\ln g_T(y)}& \ge\min_{a\in{[0,1]}}\frac{g_T'(a)}{g_T(a)}\,\posp{x-y}=\Lipm{g_T}\,\posp{x-y}\enspace,\\
    \posp{\ln g_T(x)-\ln g_T(y)}& \le\max_{a\in{[0,1]}}\frac{g_T'(a)}{g_T(a)}\,\posp{x-y}=\Lip{g_T}\,\posp{x-y}\enspace.
  \end{align*}
  Moreover, $\posp{x-y}\le\posp{\ln x-\ln y}$ if $x,y\in{]0,1]}$.
\end{proof}

\begin{lemma}
  Let $x,y\in{[0,1]}$.  If $x>y$, then $\lim_{T\to0}{g_T(x)}/{g_T(y)}=\infty$.
\end{lemma}
\begin{proof}
  This result follows directly from Lemma~\ref{lem:bornes-ln-Lmax-Lmin} and Assumption~\eqref{eq:hyp-g_T-0}.
\end{proof}

\begin{lemma}\label{lem:M_T-infty}
  The map $g_T$ tends to the constant function equal to $1$, uniformly in ${[0,1]}$, when $T$ tends to infinity.
\end{lemma}
\begin{proof}
  For all $x\in{[0,1]}$, by Lemma~\ref{lem:bornes-ln-Lmax-Lmin},
  \[
    \ln{g_T(x)}=\ln \frac{g_T(x)}{g_T(0)}\le\Lip{g_T}\,({x-0})\le\Lip{g_T}\enspace.
  \]
  Therefore, by Assumption~\eqref{eq:hyp-g_T-infty},
  \[
    \lim_{T\to\infty}\sup_{x\in{[0,1]}}\ln{g_T(x)}\le\lim_{T\to\infty}\Lip{g_T}=0\enspace.\qedhere
  \]
\end{proof}
It follows from Lemma~\ref{lem:M_T-infty} that $\lim_{T\to\infty}M_T(\vec x)=\diag(C\vun)^{-1}C$ uniformly for $\vec x\in\Sigma$.


\subsection{Fixed points and convergence of $\vec u_T$}

When $C$ is irreducible, for all $\vec x\in\Sigma$, the matrix $M_T(\vec x)$ is irreducible, and we can define the map
\[
    \vec u_T\colon \Sigma \to\Sigma \colon \vec x\mapsto \vec u_T(\vec{x})\enspace,
\]
that sends $\vec x$ to the unique invariant measure $\vec u_T(\vec{x})$ of $M_T(\vec{x})$.
The Matrix Tree Theorem enables us to express explicitly $\vec u_T(\vec x)$.
\begin{lemma}\label{lem:expr-u_T}
  Assume that $C$ is irreducible.  Then, the invariant measure of $M_T(\vec{x})$ is given by
  \[\vec u_T(\vec x)=\frac{\vec h_T(\vec x)}{\ip{\vec h_T(\vec x)}{\vun}}\enspace,\] where
  \begin{equation}\label{eq:h_T}
    \vec h_T(\vec x)_r 
    = \Big(\sum_kC_{rk}\,g_T(\vec x_k)\Big)\Big(\sum_{R\in \SpArb{r}}\prod_{(i,j)\in R}C_{ij}\,g_T(\vec x_j)\Big)\enspace.
  \end{equation}
\end{lemma}
\begin{proof}
  Apply Theorem~\ref{thm:Tutte} to $M_T(\vec x)$, and take $\vec h_T(\vec x)=\mu\,\vec v$, 
  where $\vec v$ is given by~\eqref{eq:mtt} and $\mu=\prod_i\sum_kC_{ik}g_T(\vec x_k)$.
\end{proof}

The existence of fixed points for $\vec u_T$ is then proved using Brouwer's Fixed Point Theorem.
\begin{proposition}\label{prop:u_tau-fxpt-exist}
    Assume that $C$ is irreducible.  The map $\vec u_T$ has at least one fixed point
    in $\mathrm{int}(\Sigma)$.  Moreover, every fixed point of $\vec u_T$ is in $\mathrm{int}(\Sigma)$.
\end{proposition}
\begin{proof}
    By Lemma~\ref{lem:expr-u_T}, the map $\vec u_T\colon \Sigma \to\Sigma$ is continuous, and
    therefore Brouwer's Fixed Point Theorem ensures the existence of at least one fixed point for $\vec u_T$.
    Moreover, since the invariant measure of an irreducible matrix is positive, 
    and since $M_T(\vec x)$ is irreducible,
    $\vec u_T$ sends $\Sigma$ to $\mathrm{int}(\Sigma)$, and therefore every fixed point of $\vec u_T$ is in $\mathrm{int}(\Sigma)$.
\end{proof}

The following result concerns the uniqueness of the fixed point and the convergence of the orbits of $\vec u_T$.  With Assumption~\eqref{eq:hyp-g_T-infty} satisfied, it shows that the map $\vec u_T$ has a unique fixed point and that all its orbits converge to this fixed point, for a sufficiently large temperature $T$.  It can be proved using Nussbaum's Theorem~\ref{thm:Nussbaum}, as we do it here, or, under the same hypotheses, using Krause's Theorem~\ref{thm:Krause} (take $p(\vec x)=\sum_i\vec x_i$).
\begin{theorem}\label{thm:u_tau-fxpt-conv}
    Assume that $C$ is irreducible.  If $n\Lip{g_T}\le 1$, the map $\vec u_T$ has a unique fixed point $\vec{x}_T$, which belongs to $\mathrm{int}(\Sigma)$.
    Moreover all the orbits of $\vec u_T$ converge to this fixed point.
\end{theorem}
\begin{proof}
    Since $g_T$ is increasing, $\vec h_T$ is an order-preserving map from $\R^n_{>0}$ to itself: 
    $\vec x\le\vec y$ implies $\vec h_T(\vec x)\le\vec h_T(\vec y)$.
    Now, let us show that $\vec h_T$ is \emph{subhomogeneous} on $\mathrm{int}(\Sigma)$.
    Let $\vec x\in\mathrm{int}(\Sigma)$ and $0<\lambda\le1$.
    Any entry of $\vec h_T(\vec x)$ is a sum of positively weighted terms like
    \[
      \prod_{k}g_T(\vec x_k)^{\gamma_k}\enspace,
    \]
    with $\sum_k \gamma_k=n$.
    By Lemma~\ref{lem:bornes-ln-Lmax-Lmin}, for each $k\in\set{1,\dots,n}$,
    \[
      \ln\frac{g_T(\vec x_k)}{g_T(\lambda\vec x_k)}\le \Lip{g_T}\ln\frac{1}{\lambda}\enspace.
    \]
    Therefore, if $n\Lip{g_T}\le1$, then $\lambda^{1/n}g_T(\vec x_k)\le g_T(\lambda\vec x_k)$,
    and it follows that $\lambda\vec h_T(\vec x)\le\vec h_T(\lambda\vec x)$.
    Since $0\le\vec h_T(0)$, this shows that $\vec h_T$ is subhomogeneous on $\mathrm{int}(\Sigma)$.

    Finally, let $\vec x_T$ be a fixed point of $\vec u_T$, by Proposition~\ref{prop:u_tau-fxpt-exist}.
    The derivative $\vec h_T'(\vec x)$ is a nonnegative continuous function of $\vec x$:
    \begin{align*}
      \frac{\partial \vec h_T(\vec x)_r}{\partial \vec x_\ell}
      &= C_{r\ell}\,g_T'(\vec x_\ell)\Big(\sum_{R\in \SpArb{r}}\prod_{(i,j)\in R} C_{ij}\,g_T(\vec x_j)\Big)\\
        &+\Big(\sum_kC_{rk}\,g_T(\vec x_k)\Big)
        \Big(\sum_{R\in \SpArb{r}} m_{\ell,R}\,\frac{g_T'(\vec x_\ell)}{g_T(\vec x_\ell)}
        \prod_{(i,j)\in R} C_{ij}\,g_T(\vec x_j)\Big)\enspace,
    \end{align*}
    where $m_{\ell,R}=\abs{\set{i\colon(i,\ell)\in R}}$.  Moreover, since $g_T'$ takes positive values,
    ${\partial \vec h_T(\vec x)_r}/{\partial \vec x_\ell}>0$ 
    as soon as $C_{r\ell}>0$ or $m_{\ell,R}>0$ for some $R\in \SpArb{r}$.  In particular, $m_{r,R}>0$ for all $R\in \SpArb{r}$.
    Since an irreducible matrix with positive diagonal is primitive, $\vec h_T'(\vec x)$ is also primitive
    (see for instance Corollary~$2.2.28$ in~\cite{BP94}).
    Therefore, by Theorem~\ref{thm:Nussbaum},
    $\vec x_T$ is the unique fixed point of $\vec u_T$, and all the orbits of $\vec u_T$ converges to $\vec x_T$.
\end{proof}


\subsection{Fixed points and convergence of $\vec f_T$}

We now consider the map
\[
    \vec f_T\colon \Sigma\to\Sigma \colon \vec x\mapsto \vec x M_T(\vec x)\enspace.
\]
The following result shows that if $T$ is sufficiently large, the fixed point
of the map $\vec u_T$ can be computed by iterating $\vec f_T$.

\begin{theorem}\label{thm:f_tau-fxpt-conv}
    The fixed points of $\vec u_T$ and $\vec f_T$ are the same.
    If $C$ is primitive, then, for $T$ sufficiently large,
    all the orbits of $\vec f_T$ converge
    to the fixed point $\vec x_T$ of $\vec u_T$.
\end{theorem}
\begin{proof}
    Clearly, $\vec f_T$ and $\vec u_T$ have the same fixed points.
    Suppose now that $C$ is primitive, and let us show then that, 
    for $T$ is sufficiently large, $\vec f_T$ is a contraction for some particular norm.
    For every $\vec x,\vec y\in\Sigma$ and for any norm $\norm{\cdot}$,
    \[
        \norm{\vec f_T(\vec x)-\vec f_T(\vec y)}\le\sup_{\vec v\in\Sigma}\norm{(\vec x-\vec y)\,\vec f'_T(\vec v)}\enspace.
    \]
    The derivative of $\vec f_T$ satisfies
    \[
      \frac{\partial \vec f_T(\vec v)_j}{\partial \vec v_\ell}
      = M_T(\vec v)_{\ell j} 
      + \sum_i\big(\delta_{\ell j}-M_T(\vec v)_{ij}\big)\frac{C_{i\ell}\,\vec v_i\,g_T'(\vec v_\ell)}{\sum_kC_{ik}\,g_T(\vec v_k)} \enspace.
    \]
    By Assumption~\eqref{eq:hyp-g_T-infty} and Lemma~\ref{lem:M_T-infty},
    \[
      \lim_{T\to\infty} \vec f_T'(\vec v)=\lim_{T\to\infty}M_T(\vec v)=\diag(C\vun)^{-1}C\enspace,
    \]
    uniformly for $\vec v\in\Sigma$.  Let $A=\diag(C\vun)^{-1}C$, and 
    let $\mathcal{S}=\set{\vec z\in\R^n\colon\ip{\vec z}{\vun}=0}$ be the space of vectors orthogonal to the vector $\vun$.
    The map $\vec x\mapsto \vec xA$ preserves the space $\mathcal{S}$, because $A\vun=\vun$. 
    Let $A_{\mathcal{S}}$ denote the restriction of the map $\vec x\mapsto \vec xA$ to $\mathcal{S}$.
    Since the matrix $A$ is primitive, its spectral radius, $\rho(A)$, is a simple eigenvalue 
    and all the other eigenvalues of $A$ have a strictly smaller modulus. 
    Moreover, the space $\mathcal{S}$ contains no eigenvector of $A$ for the eigenvalue $\rho(A)$, 
    because such an eigenvector $\vec{u}$ must be a scalar multiple of the Perron eigenvector of $A$, 
    which has positive entries, contradicting $\ip{\vec u}{\vun} =0$.
    We deduce that $\rho(A_{\mathcal S})<\rho(A)=1$. 
    It follows that there exists a norm $\norm{\cdot}$ such that $\indnorm{A_{\mathcal S}}<1$,
    where $\indnorm{\cdot}$ is the matrix norm induced by $\norm{\cdot}$ (see for instance Lemma~$5.6.10$ in~\cite{HoJo}).
    Therefore, since $\vec f'_T(\vec v)$ tends uniformly to $A$ for $\vec v\in\Sigma$ when $T$ tends to $\infty$,
    \[
      \lim_{T\to\infty}\sup_{\vec v\in\Sigma}\indnorm{\vec f_T'(\vec v)_\mathcal{S}}=\indnorm{A_\mathcal{S}}<1\enspace,
    \]   
    where $\vec f'_T(\vec v)_\mathcal{S}$ denotes the restriction of $\vec f'_T(\vec v)$ on $\mathcal{S}$.
    It follows that, for all $\alpha\in{]\indnorm{A_\mathcal{S}},1[}$, there exists $T_\alpha$ such that for all $T>T_\alpha$,
    \[
        \norm{\vec f_T(\vec x)-\vec f_T(\vec y)}
        \le\norm{\vec x-\vec y}\,\sup_{\vec v\in\Sigma}\indnorm{\vec f'_T(\vec v)_\mathcal{S}}\le\alpha\norm{\vec x-\vec y}\enspace.
    \]
    Hence, for such temperature $T$, by Banach's Fixed Point Theorem,
    $\vec f_T$ has a unique fixed point and every orbit of $\vec f_T$ converges to this fixed point.
\end{proof}

\paragraph{Example}
    For the convergence of the orbits of $\vec f_T$, the primitivity of $C$ cannot be dispensed with.
    Let for instance
    $C=\left(\begin{smallmatrix}0&1\\1&0\end{smallmatrix}\right)$
    and $T>0$.
    Then $M_T(\vec x)=C$ for all $\vec x\in\Sigma$, and $\vec f_T^k(\vec x)$ oscillates when $k$ tends to infinity,
    unless $\vec x=\left(\frac{1}{2}\,\,\frac{1}{2}\right)$.

For positive matrices $C\in\R^{n\times n}_{>0}$, let us now derive another convergence criterion,
depending on Birkhoff's coefficient of ergodicity.
\begin{lemma}\label{lem:d_H(M)-d_H(x)}
    Assume that $C$ is positive.  Then, for any $\vec x,\vec y\in\mathrm{int}(\Sigma)$,
    \[
        d_H(M_T(\vec x),M_T(\vec y))\le 2\Lip{g_T}\,d_H(\vec x,\vec y)\enspace.
    \]
\end{lemma}
\begin{proof}
    Let $\vec x,\vec y\in\mathrm{int}(\Sigma)$ be fixed.  
    Let us define $\alpha=\max_{i,j}\frac{M_T(\vec x)_{ij}}{M_T(\vec y)_{ij}}$
    and $\beta=\min_{i,j}\frac{M_T(\vec x)_{ij}}{M_T(\vec y)_{ij}}$.
    By definition,
    \begin{align*}
        d_H(M_T(\vec x),M_T(\vec y))
        &=\sup_{\vec z\in\R^n_{>0}}\max_{i,j}
          \ln\frac{(\vec zM_T(\vec x))_i}{(\vec zM_T(\vec y))_i}\frac{(\vec zM_T(\vec y))_j}{(\vec zM_T(\vec x))_j}\\
        &\le\sup_{\vec z\in\R^n_{>0}}\max_{i,j}\ln
        \frac{(\alpha \vec zM_T(\vec y))_i}{(\vec zM_T(\vec y))_i}
          \frac{(\vec zM_T(\vec y))_j}{(\beta \vec zM_T(\vec y))_j}
        =\ln\frac{\alpha}{\beta}\enspace.
    \end{align*}
    Moreover, by Lemma~\ref{lem:bornes-ln-Lmax-Lmin},
    \begin{align*}
        \ln\alpha
        &=\max_{i,j}\ln\bigg(\frac{C_{ij}\,g_T(\vec x_j)}{\sum_k C_{ik}\,g_T(\vec x_k)}
             \frac{\sum_k C_{ik}\,g_T(\vec y_k)}{C_{ij}\,g_T(\vec y_j)}\bigg)\\
        &\le\max_{j,k}\left(\ln\frac{g_T(\vec x_j)}{g_T(\vec y_j)}+\ln\frac{g_T(\vec y_k)}{g_T(\vec x_k)}\right)\\
        &\le\Lip{g_T}\left(\max_{j}\bposp{\ln\frac{\vec x_j}{\vec y_j}}+\max_{k}\bposp{\ln\frac{\vec y_k}{\vec x_k}}\right)\\
        &=\Lip{g_T}\left(\bposp{\max_{j}\ln\frac{\vec x_j}{\vec y_j}}+\bposp{\max_{k}\ln\frac{\vec y_k}{\vec x_k}}\right)\\
        &=\Lip{g_T}\,d_H(\vec x,\vec y)\enspace,
    \end{align*}
    since $\vec x,\vec y\in\mathrm{int}(\Sigma)$ implies $\max_{j}\ln\frac{\vec x_j}{\vec y_j}\ge0$ and $\max_{k}\ln\frac{\vec y_k}{\vec x_k}\ge0$.
    We get similarly $-\ln\beta\le\Lip{g_T}\,d_H(\vec x,\vec y)$.
\end{proof}
\begin{proposition}\label{prop:Cpos-f_Tcontr}
    Assume that $C$ is positive.
    If $2\Lip{g_T}<{1-\tauB(C)}$,
    then $\vec f_T$ has a unique fixed point $\vec x_T\in\mathrm{int}(\Sigma)$ and
    all the orbits of $\vec f_T$ converge to this fixed point.
\end{proposition}
\begin{proof}
    Let $\vec x,\vec y\in\mathrm{int}(\Sigma)$.  By Lemma~\ref{lem:d_H(M)-d_H(x)},
    \begin{align*}
        d_H(\vec f_T(\vec x),\vec f_T(\vec y))
        &\le d_H(\vec xM_T(\vec x),\vec yM_T(\vec x))+d_H(\vec yM_T(\vec x),\vec yM_T(\vec y)) \\
        &\le \tauB(M_T(\vec x))\,d_H(\vec x,\vec y)+d_H(M_T(\vec x),M_T(\vec y)) \\
        &\le \big(\tauB(C)+ 2\Lip{g_T}\big)\,d_H(\vec x,\vec y)\enspace.
    \end{align*}
    Therefore, if $2\Lip{g_T}<{1-\tauB(C)}$, then $\vec f_T$ is a contraction on $\mathrm{int}(\Sigma)$ 
    with respect to the distance $d_H$.
    By Banach's Fixed Point Theorem, $\vec f_T$ has a unique fixed point $\vec x_T\in\mathrm{int}(\Sigma)$ and
    all the orbits of $\vec f_T$ converge to this fixed point.
\end{proof}


\subsection{Existence of multiple fixed points of $\vec u_T$ and $\vec f_T$}
Theorems~\ref{thm:u_tau-fxpt-conv} and~\ref{thm:f_tau-fxpt-conv} show that for a sufficiently large temperature $T$, the maps $\vec u_T$ and $\vec f_T$ have a unique fixed point.  We can naturally wonder about the uniqueness of the fixed point of these maps for small $T$: we show that, at least when $C$ is positive, multiple fixed points always exist.

\begin{theorem}\label{thm:mult-pt-fx}
    Assume that $C$ is irreducible and that the first column of $C$ is positive.
    Then, for all $0<\eps<\frac{1}{2}$, there exists $T_\eps$ such that for $T\le T_\eps$,
    the map $\vec u_T$ has a fixed point in $\Sigma_\eps=\set{\vec x\in\Sigma, \vec x_1\ge1-\eps}$.

  Assume now that $C$ is irreducible with $C_{11}>0$ only, 
  and that there exists $\eps_n>0$, independent of $T$, such that 
  $g_T(\eps_n)^{n-1}\le g_T(1-\eps_n)$ for all $T>0$.
  Then, for all $0<\eps<\eps_n$, there exists $T_\eps$ such  that for $T\le T_\eps$,
  the map $\vec u_T$ has a fixed point in $\Sigma_\eps=\set{\vec x\in\Sigma, \vec x_1\ge1-\eps}$.
\end{theorem}
\begin{proof}
  Let $k$ be the number of indices $i\neq1$ such that $C_{i1}>0$.
  Let $0<\eps<\frac{1}{2}$, and let $\vec x\in\Sigma_\eps$.
  By the irreducibility of $C$, there exists a spanning arborescence $R$ rooted at $1$,
  containing all the $k$ arcs $(i,1)$ with $i\neq1$ and $C_{i1}>0$.
  Hence
  \[
    \vec h_T(\vec x)_1 \ge C_{11} g_T(\vec x_1) \prod_{(i,j)\in R} C_{ij}g_T(\vec x_j)
    \ge \alpha g_T(\vec x_1)^{k+1}\ge \alpha g_T(\vec x_1)^{k}g_T(1-\eps)\enspace,
  \]
  where $\alpha=C_{11}\prod_{(i,j)\in R} C_{ij}>0$.
  Let $r\neq1$.  If $C_{r1}\neq0$, then a spanning arborescence rooted at $r$ can have
  at most $k-1$ arcs $(i,1)$ with $C_{i1}>0$, 
  whereas it can have at most $k$  arcs $(i,1)$ with $C_{i1}>0$ in general.
  Hence, in all cases, $\vec h_T(\vec x)_r$ is a sum of positively weighted terms like
  $\prod_\ell g_T(\vec x_\ell)^{\gamma_\ell}$ with $\sum_\ell\gamma_\ell=n$
  and $\gamma_1\le k$.  This implies that, for $r\neq1$,
  \[
    \vec h_T(\vec x)_r \le \beta g_T(\vec x_1)^k g_T(\eps)^{n-k}\enspace,
  \]
  for some positive constant $\beta$.
  Therefore,
  \begin{align*}
    \vec u_T(\vec x)_1 &= \frac{1}{1+\sum_{r\neq1}\frac{\vec h_T(\vec x)_r}{\vec h_T(\vec x)_1}}
      \ge \frac{1}{1+(n-1)\frac{\beta}{\alpha}\frac{g_T(\eps)^{n-k}}{g_T(1-\eps)}}\enspace.
  \end{align*}

  If the first column of $C$ is positive, then $k=n-1$.
  By Lemma~\ref{lem:bornes-ln-Lmax-Lmin}, 
  \[
    \ln\frac{g_T(1-\eps)}{g_T(\eps)}\ge\Lipm{g_T}\,(1-2\eps)\enspace.
  \]
  Therefore, if $\Lipm{g_T}(1-2\eps)\ge\ln\frac{(n-1)\beta}{\alpha}\frac{1-\eps}{\eps}$, we get
  $\vec u_T(\vec x)_1\ge 1-\eps$.  This shows that $\vec u_T(\Sigma_\eps)\subset\Sigma_\eps$.
  By Brouwer's Fixed Point Theorem, 
  the continuous map $\vec u_T$ has therefore at least one fixed point in $\Sigma_\eps$.

  Now, suppose we know only that $C_{11}>0$, but there exists $\eps_n>0$ such that 
  $g_T(\eps_n)^{n-1}\le g_T(1-\eps_n)$ for all $T>0$.
  The map $\varphi_T\colon\eps\mapsto(n-1)\ln g_T(\eps)-\ln g_T(1-\eps)$
  is increasing and its derivative satisfies $\varphi_T'(\eps)\ge n \Lipm{g_T}$.
  Let $0<\eps<\eps_n$.  We have $\varphi_T(\eps_n)-\varphi_T(\eps)\ge n\Lipm{g_T}(\eps_n-\eps)$.
  Moreover, $k\ge1$ by irreducibility of $C$ and $\varphi_T(\eps_n)\le0$, hence
  \[
    \ln\frac{g_T(1-\eps)}{g_T(\eps)^{n-k}}
    \ge\ln\frac{g_T(1-\eps)}{g_T(\eps)^{n-1}} = {-\varphi_T(\eps)}
    \ge {n\Lipm{g_T}(\eps_n-\eps)}\enspace.
  \]
  Therefore, if $n\Lipm{g_T}(\eps_n-\eps)\ge\ln\frac{(n-1)\beta}{\alpha}\frac{1-\eps}{\eps}$,
  we get $\vec u_T(\vec x)_1\ge 1-\eps$.  The result follows by the same argument as above.
\end{proof}

\begin{remark}  When $g_T(x)=\e{E(x)/T}$ for some increasing energy $E$,
  then $\eps_n$ satisfies the condition $g_T(\eps_n)^{n-1}\le g_T(1-\eps_n)$ for all $T>0$
  if and only if $(n-1)E(\eps_n)\le E(1-\eps_n)$, which holds for some $0<\eps_n<1$,
  since $E(0)=0$ and $E(1)>0$.
\end{remark}

\begin{corollary}
  If $C$ is positive, then, for $T>0$ sufficiently small,
  the map $\vec u_T$ has several fixed points in~$\Sigma$.
\end{corollary}

\paragraph{Example}
If $C$ is not positive, the existence of several fixed points for small $T$ is not insured.
    Indeed, we shall see in Remark~\ref{rem:matrices22} that for
    $C=\left(\begin{smallmatrix}1&2\\1&0\end{smallmatrix}\right)$ and $g_T(x)=\e{x/T}$,
    the fixed point of $\vec u_T$ and $\vec f_T$ is unique for each $T>0$.


\section{Refinement of the model}\label{sec:refinement}

In the present section, we study a more general
model, which includes a \mdef{damping factor} $0<\gamma<1$,
as is the standard definition of Google's PageRank~\cite{BP98},
in which the websurfer either jumps to the search engine
with probability $1-\gamma$ or moves to a neighbor page
with probability $\gamma$. The presence of a
damping factor yields a more realistic model of the
websurfer walk. Moreover, it allows one to deal with reducible matrices,
and it improves the convergence speed of iterative methods.

Let $C$ be a $n\times n$ nonnegative matrix with no zero row and let $\vec d\in\R^n_{>0}$ be a \mdef{personalization vector}.
For all temperature $0<T<\infty$, let us define as previously $g_T\colon {[0,1]}\to\R_{>0}$ as a continuously differentiable and increasing map, with $g_T(0)=1$ and $g_T'\colon {[0,1]}\to\R_{>0}$.  Suppose that Assumptions~\eqref{eq:hyp-g_T-0} and~\eqref{eq:hyp-g_T-infty} are satisfied.  For a temperature $T=\infty$, let us also define $g_\infty(x)=1$ for all $x\in{[0,1]}$.

For any two temperatures $0<T_1,T_2\le\infty$, and for all $\vec x\in\Sigma$, we can consider the positive transition matrix $M_{T_1,T_2,\gamma}(\vec x)$ defined as
\begin{equation}\label{eq:M-T1-T2-gamma}
    M_{T_1,T_2,\gamma}(\vec x)_{ij} 
    = \gamma\, \frac{C_{ij}\,g_{T_1}(\vec x_j)}{\sum_k C_{ik}\,g_{T_1}(\vec x_k)}
    + (1-\gamma)\,\frac{\vec d_j\, g_{T_2}(\vec x_j)}{\sum_k\vec d_k\, g_{T_2}(\vec x_k)}\enspace.
\end{equation}

\begin{remark}
  For simplicity, we consider the same family of weight functions $g_T$
  for the first and the second term of $M_{T_1,T_2,\gamma}(\vec x)$.
  Note however that the results of this section remain true 
  if two families $g_{T_1}$ and $\tilde g_{T_2}$ are considered.
\end{remark}

\begin{remark}
Suppose $T_1=\infty$, $T_2<\infty$ and $0<\gamma<1$, and let $\vec x$ be the current ranking vector.  Then $M_{T_1,T_2,\gamma}(\vec x)$ is the transition matrix of the following random walk on the graph. At each step of his walk, either, with probability $\gamma$, the websurfer
draws the next page uniformly among the pages referenced by his current page.  Or, with probability $1-\gamma$, he refers to the Web search engine, and therefore preferentially chooses for the next page a webpage with a good ranking.
\end{remark}

The maps $\vec u_{T_1,T_2,\gamma}$ and $\vec f_{T_1,T_2,\gamma}$ are defined as previously: $\vec u_{T_1,T_2,\gamma}(\vec x)$ is the unique invariant measure of $M_{T_1,T_2,\gamma}(\vec x)$ and $\vec f_{T_1,T_2,\gamma}(\vec x)=\vec xM_{T_1,T_2,\gamma}(\vec x)$.

Theorem~\ref{thm:u_tau-fxpt-conv} about the uniqueness of the fixed point of $\vec u_{T}$ can be adapted in the following way.
\begin{proposition}
    If $n\Lip{g_{T_1}}+(n-1)\Lip{ g_{T_2}}\le1$,
    the map $\vec u_{T_1,T_2,\gamma}$ has a unique fixed point
    $\vec x_{T_1,T_2,\gamma}$ in $\Sigma$.  
    Moreover, all the orbits of $\vec u_{T_1,T_2,\gamma}$ converge
    to the fixed point $\vec x_{T_1,T_2,\gamma}$.
\end{proposition}
\begin{proof}
  If $T_1=T_2=\infty$, the result follows directly from Perron--Frobenius theory.
  Let us therefore suppose that $T_1<\infty$ or $T_2<\infty$.
  For every $\vec x\in\Sigma$, by Theorem~\ref{thm:Tutte},
  $\vec u_{T_1,T_2,\gamma}={\vec h_{T_1,T_2,\gamma}(\vec x)}/{\ip{\vec h_{T_1,T_2,\gamma}(\vec x)}{\vun}}$,
  where
  \[
    \vec h_{T_1,T_2,\gamma}(\vec x)_r 
    = \Big(\sum_kC_{rk}\,g_{T_1}(\vec x_k)\Big) \\
    \Big(\sum_{R\in \SpArb{r}}\prod_{(i,j)\in R}W(\vec x)_{ij}\Big)\enspace,
  \]
  with
  $W(\vec x)_{ij}=\sum_k\big(
    \gamma\, C_{ij}\,g_{T_1}(\vec x_j)\,\vec d_k\, g_{T_2}(\vec x_k)+(1-\gamma)\,C_{ik}\,g_{T_1}(\vec x_k)\,\vec d_j\, g_{T_2}(\vec x_j)\big)$.

  Since $g_{T_1}$ and $g_{T_2}$ are nondecreasing, $\vec h_{T_1,T_2,\gamma}$ is an order-preserving map.
  Moreover, assume that $n\Lip{g_{T_1}}+(n-1)\Lip{ g_{T_1}}\le1$.
  Then, as in the proof of Theorem~\ref{thm:u_tau-fxpt-conv},
  $\vec h_{T_1,T_2,\gamma}(\vec x)$ is shown to be subhomogeneous on $\mathrm{int}(\Sigma)$.

  Finally, the derivative $\vec h_{T_1,T_2,\gamma}'(\vec x)$ is a nonnegative continuous function:
  \begin{multline*}
    \frac{\partial \vec h_{T_1,T_2,\gamma}(\vec x)_r}{\partial \vec x_\ell}
    = C_{r\ell}\,g_{T_1}'(\vec x_\ell)\Big(\sum_{R\in \SpArb{r}}\prod_{(i,j)\in R}W(\vec x)_{ij}\Big)\\
    + \Big(\sum_kC_{rk}\,g_{T_1}(\vec x_k)\Big)
    \bigg(\sum_{R\in \SpArb{r}}\Big(\prod_{(i,j)\in R}W(\vec x)_{ij}\Big)
    \Big(\sum_{(i,j)\in R}\frac{\partial W(\vec x)_{ij}/\partial \vec x_\ell}{W(\vec x)_{ij}}\Big)\bigg)\enspace,
  \end{multline*}
  where
  \begin{multline*}
    \frac{\partial W(\vec x)_{ij}}{\partial \vec x_\ell}
    =\delta_{\ell j}\sum_k
    \big(\gamma\, C_{i\ell}\,g_{T_1}'(\vec x_\ell)\,\vec d_k\, g_{T_2}(\vec x_k)
    +(1-\gamma)\,C_{ik}\,g_{T_1}(\vec x_k)\,\vec d_\ell\, g_{T_2}'(\vec x_\ell)\big)\\
    +\gamma\, C_{ij}\,g_{T_1}(\vec x_j)\,\vec d_\ell\, g_{T_2}'(\vec x_\ell)
    +(1-\gamma)\,C_{i\ell}\,g_{T_1}'(\vec x_\ell)\,\vec d_j\, g_{T_2}(\vec x_j)\enspace.
  \end{multline*}
  Let us now prove that $\vec h_{T_1,T_2,\gamma}'(\vec x)$ is a positive matrix for every $\vec x$.
  Suppose first that ${T_2}<\infty$.  Then $ g_{T_2}'(\vec x_\ell)>0$, and
  there exists a spanning arborescence $R\in \SpArb{r}$ and a node $i$ such that $(i,\ell)\in R$, 
  since $\Gr{M_{T_1,T_2,\gamma}(\vec x)}$ is the complete graph.
  It follows that, for this $R$,
  \[
    \sum_{(i,j)\in R}\frac{\partial W(\vec x)_{ij}/\partial \vec x_\ell}{W(\vec x)_{ij}}
    \ge\frac{\partial W(\vec x)_{i\ell}/\partial \vec x_\ell}{W(\vec x)_{i\ell}}
    \ge\frac{\sum_k(1-\gamma)\,C_{ik}\,g_{T_1}(\vec x_k)\,\vec d_\ell\, g_{T_2}'(\vec x_\ell)}{W(\vec x)_{i\ell}}>0\enspace,
  \]
  and hence ${\partial \vec h_{T_1,T_2,\gamma}(\vec x)_r}/{\partial \vec x_\ell}>0$.  
  Now, suppose that ${T_2}=\infty$ and ${T_1}<\infty$.  Then we can suppose without loss of generality that $C$ has no zero column (see Remark~\ref{rem:zero-row} below).
  Either $C_{r\ell}>0$, and therefore ${\partial \vec h_{T_1}(\vec x)_r}/{\partial \vec x_\ell}>0$.
  Or there exists $i\neq r$ such that $C_{i\ell}>0$, and for all $R\in \SpArb{r}$, there exists $j$ such that $(i,j)\in R$, that is
  \[
    \frac{\partial W(\vec x)_{ij}}{\partial \vec x_\ell}
    \ge(1-\gamma)\,C_{i\ell}\,\vec d_j\,g_{T_1}'(\vec x_\ell)>0\enspace,
  \]
  and hence ${\partial \vec h_{T_1,T_2,\gamma}(\vec x)_r}/{\partial \vec x_\ell}>0$. 
  
  Since Brouwer's Fixed Point Theorem ensures the existence of at least one fixed point $\vec x_{T_1,T_2,\gamma}\in\mathrm{int}(\Sigma)$ 
  for the continuous map $\vec u_{T_1,T_2,\gamma}$ which sends $\Sigma$ to $\mathrm{int}(\Sigma)$, by Theorem~\ref{thm:Nussbaum},
  this fixed point $\vec x_{T_1,T_2,\gamma}$ is the unique fixed point of $\vec u_{T_1,T_2,\gamma}$, 
  and all the orbits of $\vec u_{T_1,T_2,\gamma}$ converge to $\vec x_{T_1,T_2,\gamma}$.
\end{proof}

\begin{remark}\label{rem:zero-row}
  If $T_2=\infty$ and the matrix $C$ has a zero column, 
  the problem can be reduced to a problem of smaller dimension with a matrix with no zero column.
  Indeed, suppose the $n$th column of $C$ is zero.  Then
  \begin{align*}
    (\vec u_{T_1,\infty,\gamma}(\vec x)_1\,\cdots\,\vec u_{T_1,\infty,\gamma}(\vec x)_{n-1}) 
      &=\frac{\sum_k\vec d_k-(1-\gamma)\vec d_n}{\sum_k\vec d_k}\,\vec{\tilde u}_{T_1,\infty,\gamma}(\vec{\tilde x})\enspace,\\
    \vec u_{T_1,\infty,\gamma}(\vec x)_n&=(1-\gamma)\frac{\vec d_n}{\sum_k\vec d_k} \enspace,
  \end{align*}
  where $\vec{\tilde x}=(\vec x_1 \,\cdots\,\vec x_{n-1})$ and $\vec{\tilde u}_{T_1,\infty,\gamma}(\vec{\tilde x})$ 
  is the invariant measure of a matrix $\tilde M_{T_1,\infty,\gamma}(\vec{\tilde x})$, 
  with $\tilde C$ the principal submatrix of $C$ corresponding to the indices $1,\dots,n-1$, and $\vec{\tilde d}$ some positive vector of length $n-1$.
\end{remark}

The following adaptations of Theorem~\ref{thm:f_tau-fxpt-conv} and Proposition~\ref{prop:Cpos-f_Tcontr} about the uniqueness of the fixed point of $\vec f_{T}$ are quite direct.
\begin{proposition}
    The fixed points of $\vec u_{T_1,T_2,\gamma}$ and $\vec f_{T_1,T_2,\gamma}$ are the same.
    Moreover, for ${T_1}$ and ${T_2}$ sufficiently large,
    all the orbits of $\vec f_{T_1,T_2,\gamma}$ converge
    to the fixed point $\vec x_{T_1,T_2,\gamma}$ of $\vec u_{T_1,T_2,\gamma}$.
\end{proposition}
\begin{proposition}
    Assume that $C$ is positive.  If $2(\Lip{g_{T_1}}+\Lip{ g_{T_2}})<{1-\tauB(C)}$,
    then $\vec f_{T_1,T_2,\gamma}$ has a unique fixed point $\vec x_{T_1,T_2,\gamma}\in\mathrm{int}(\Sigma)$ and
    all the orbits of $\vec f_{T_1,T_2,\gamma}$ converge to the fixed point $\vec x_{T_1,T_2,\gamma}$.
\end{proposition}

We next show that the map $\vec f_{T_1,T_2,\gamma}$ has multiple fixed points
if either $T_1$ or $T_2$ is sufficiently small. Of course, this will
require the damping factor to give enough weight to the terms
corresponding to the small temperature in equation~\eqref{eq:M-T1-T2-gamma}.
\begin{proposition}\label{prop:T1small}
    Assume that the first column of $C$ is positive and that $\frac{1}{2}<\gamma<1$.
    Then, there exists $T_0>0$ such that, for all $T_1\le T_0$ and for all $T_2\in {[0,\infty]}$,
    the map $\vec f_{T_1,T_2,\gamma}$ has a fixed point in
    $\set{\vec x\in\Sigma\colon \vec x_1>\frac{1}{2}}$.
\end{proposition}

The conclusion of Proposition~\ref{prop:T1small} is weaker than that
of Theorem~\ref{thm:mult-pt-fx}. The latter shows that for a sufficiently small temperature,
we can find a fixed point of $\vec f_{T}$ arbitrarily close to a vertex
of the simplex, whereas the former shows that for a sufficiently small
temperature, we can find a fixed point of $\vec f_{T_1,T_2,\gamma}$ in the region $\vec x_1>1/2$
of the simplex. In fact, such a fixed point may not approach a vertex
of the simplex as the temperature tends to~$0$. This discrepancy
is due to presence of the damping factor.
The proof of Proposition.~\ref{prop:T1small} is a straightforward adaptation of the
proof of Theorem~\ref{thm:mult-pt-fx}.

\begin{remark}
    If the first column of $C$ is not positive,
    the existence of a fixed point such that $\vec x_1>\vec x_2,\dots,\vec x_n$
    for small ${T_1}$ is not insured.
    Consider for instance
    $C=\left(\begin{smallmatrix}0&1\\1&1\end{smallmatrix}\right)$,
    $g_T(x)=\e{x/T}$, ${T_2}=\infty$, $\vec d=(1\,\,1)$ and $0<\gamma<1$.
    Then, for each each ${T_1}>0$,
    any fixed point $\vec x_{{T_1},\infty,\gamma}$ of $\vec f_{{T_1},\infty,\gamma}$
    belongs to ${[0,\frac{1}{2}[}\times{]\frac{1}{2},1]}$.
\end{remark}

We now show the existence of multiple fixed points
for a sufficiently small temperature $T_2$.
\begin{proposition}
  Let $0<\eps<\frac{1}{2}$ and let $\mu=\sum_{k\neq1}{\vec d_k}/{\vec d_1}$.
  If $0<\gamma<\eps$, there exists $T_{\eps,\gamma}$ such that for $T_2\le T_{\eps,\gamma}$
  and for all $T_1\in {[0,\infty]}$,
  the map $\vec f_{\infty,{T_2},\gamma}$ has a fixed point in
  $\Sigma_\eps=\set{\vec x\in\Sigma\colon \vec x_1\ge1-\eps}$.
\end{proposition}
\begin{proof}
  Let $\vec x\in \Sigma_\eps$.
  Since by Lemma~\ref{lem:bornes-ln-Lmax-Lmin},
  $\ln({ g_{T_2}(1-\eps)}/{ g_{T_2}(\eps)})\ge\Lipm{ g_{T_2}}\,(1-2\eps)$,
  we have
  \begin{align*}
    \vec f_{T_1,T_2,\gamma}(\vec x)_1 
    &= \gamma\sum_i\frac{C_{i1}g_{T1}(\vec x_1)\vec x_i}{\sum_{k}C_{ik}g_{T1}(\vec x_k)}
    +(1-\gamma)\frac{\vec d_1 g_{T_2}(\vec x_1)}{\sum_k\vec d_k g_{T_2}(\vec x_k)}
    \ge (1-\gamma)\frac{\vec d_1\, g_{T_2}(\vec x_1)}{\sum_k\vec d_k\, g_{T_2}(\vec x_k)}\\
    &\ge\frac{1-\gamma}{1+\mu\,\e{-\Lipm{ g_{T_2}}\,(1-2\eps)}}\enspace.
  \end{align*}
  Therefore, if $\Lipm{ g_{T_2}}(1-2\eps)\ge\ln\frac{\mu(1-\eps)}{\eps-\gamma}$,
  then $\vec f_{T_1,T_2,\gamma}(\vec x)_1\ge 1-\eps$
  and $\vec f_{\infty,T_2,\gamma}(\vec x)\in  \Sigma_\eps$.
  By Brouwer's Fixed Point Theorem, the continuous map $\vec f_{T_1,T_2,\gamma}$
  has at least one fixed point in $\Sigma_\eps$.
\end{proof}

We now consider the case where the damping factor $\gamma$ approaches $1$ and $T_1=\infty$. Then, the
corresponding value of the generalized PageRank converges to an invariant measure of the matrix $\diag(C\vun)^{-1}C$, independently of the choice of $T_2$. In order to prove this,  we will need the following classical result, which is a particular case of Corollary~3.1 in~\cite{Mey74}.  For a matrix $M\in\C^{n\times n}$, the \mdef{index}, $\ind{M}$, is the smallest nonnegative integer $k$ such that $\rank{M^{k+1}}=\rank{M^k}$, and its \mdef{Drazin inverse}, $M^D$, is the unique solution $X$ of the equations $M^{k+1}X=M^k$, $XMX=X$, $MX=XM$, where $k=\ind{M}$.
\begin{lemma}[See~\cite{Mey74}, Coro.~3.1]\label{lem:M-matrices}
  Let $M\in\C^{n\times n}$.  If $\ind{M}=1$, then $\lim_{\eps\to0}\eps(M+\eps I)^{-1}=I-MM^D$.
\end{lemma}

\begin{proposition}
  Assume $T_1=\infty$.
  For all vector norm $\norm{\cdot}$ and all $\eps>0$, there exists $\gamma_\eps<1$ such that
  for every fixed point $\vec x$ of $\vec f_{\infty,T_2,\gamma}$, with $\gamma_\eps<\gamma<1$, 
  there exists an invariant measure $\vec u$ of $\diag(C\vun)^{-1}C$ such that $\norm{\vec x-\vec u}<\eps$.
\end{proposition}
\begin{proof}
  Let $A=\diag(C\vun)^{-1}C$, and let $M=I-A$.
  Since $A$ is stochastic, $\ind{M}=1$ (see for instance Theorem~8.4.2 in~\cite{BP94}).
  Therefore, with $\gamma=(1+\eps)^{-1}$,
  \[
    \lim_{\gamma\to1}(1-\gamma)(I-\gamma A)^{-1}=\lim_{\eps\to0}\eps(M+\eps I)^{-1}=I-MM^D\enspace.
  \]
  Let $\norm{\cdot}$ be a vector norm and $\indnorm{\cdot}$ its induced matrix norm,
  and let $\nu>0$ such that $\norm{\vec v}\le\nu$ for all stochastic vector $\vec v$.
  Let $\eps>0$.  There exists $\gamma_\eps<1$ such that if $\gamma_\eps<\gamma<1$,
  \[
    \indnorm{(1-\gamma)(I-\gamma A)^{-1}-(I-MM^D)}<\nu^{-1}\eps\enspace.
  \]
  Let $\gamma\in{]\gamma_\eps,1[}$, and let $\vec x$ be a fixed point of $\vec f_{\infty,T_2,\gamma}$,
  that is $\vec x=\vec v(\vec x)(1-\gamma)(I-\gamma A)^{-1}$,
  where $\vec v(\vec x)_i={\vec d_i\,g_{T_2}(\vec x_i)}/{\sum_k \vec d_k\,g_{T_2}(\vec x_k)}$
  for all $i$.  Then,
  \[ \norm{\vec x-\vec v(\vec x)(I-MM^D)}<\eps\enspace. \]
  But $\vec v(\vec x)(I-MM^D)$ is an invariant measure of the matrix $A$.
  Indeed, $I-MM^D$ is stochastic, 
  and $(I-MM^D)(I-A)=M-MM^DM=M-M^2M^D=0$, by definition of the Drazin inverse.
\end{proof}


\section{Estimating the critical temperature}\label{sec:particular-cases}

We call \mdef{critical temperature} the larger temperature for which the number of fixed points of $\vec u_T$ changes.  It corresponds to the loss of the uniqueness of the fixed point.  In this section, we are interested in estimating the critical temperature for some particular cases.  We will study in details the case of $n\times n$ matrices of all ones with the particular weight function $g_T(x)=\e{x/T}$.


We suppose that $g_T(x)=\e{x/T}$ and first consider the particular case where the graph is \emph{complete} with 
\[
    C=\begin{pmatrix}1&\cdots&1\\
    \vdots&\ddots&\vdots\\
    1&\cdots&1\end{pmatrix}\in\R^{n\times n}\enspace.
\]
For this matrix, the point $\vec x=\frac{1}{n}\vun$ is a fixed point of $\vec u_T$ for all $T$.
We are interested in the existence of other fixed points, depending on the temperature $T$.
\begin{lemma}\label{lem:x-e(-x/T)}
    Assume that $C$ is the $n\times n$ matrix of all ones.
    The point $\vec x$ is a fixed point of $\vec u_T$ if and only if
    $\vec x\in\Sigma$ and there exists $\lambda\in\R$ such that
    $\lambda={\vec x_i}{\e{-\vec x_i/T}}$ for every $i=1,\dots,n$.
\end{lemma}
\begin{proof}
    This follows directly from $\vec x=\vec f_T(\vec x)$.
\end{proof}

\begin{lemma}\label{lem:all-ones}
    Assume that $C$ is the $n\times n$ matrix of all ones.
    The point $\vec x$ is a fixed point of $\vec u_T$ if and only if
    $\vec x\in\Sigma$ and there exists $\mathcal{K}\subseteq\set{1,\dots,n}$,
    $y\in{[0,T]}$ and $z\ge T$ such that $y\e{-y/T}=z\e{-z/T}$, 
    $|\mathcal{K}|y+(n-|\mathcal{K}|)z=1$,
    $\vec x_i = y$ for all $i\in\mathcal{K}$ and 
    $\vec x_i = z$ for all $i\notin\mathcal{K}$.
\end{lemma}
\begin{proof}
    Since the map $x\mapsto x\e{-x/T}$ is increasing for $0\le x<T$ and decreasing for $x>T$,
    there can be at most two values $y\neq z$ such that $y\e{-y/T}=z\e{-z/T}=\lambda$ for a given $\lambda\in\R$.
    The result hence follows from Lemma~\ref{lem:x-e(-x/T)} and $\vec x\in\Sigma$.
\end{proof}

Since in our case, $\Lip{g_T}=T^{-1}$, we know from Theorem~\ref{thm:u_tau-fxpt-conv} that the critical temperature is at most $n$.  Proposition~\ref{prop:allones-bornes} shows that this critical temperature is in fact roughly $(\ln n)^{-1}$ when $n$ tends to infinity.
\begin{proposition}\label{prop:allones-bornes}
    Assume that $C$ is the $n\times n$ matrix of all ones.
    If $n>2$, the map $\vec u_T$ has a unique fixed point 
    $\vec x=\left(\frac{1}{n}\,\cdots\,\frac{1}{n}\right)$
    if and only if $T> T^*(n)$, where
    \[
        \frac{1-\frac{1}{\ln n}}{\ln((\ln n-1)n+1)}
        \le T^*(n)
        = \sup_{\alpha>1}\frac{1-\frac{1}{\alpha}}{\ln((\alpha-1)n+1)}
        < \frac{1}{\ln{(n-1)}}\enspace,
    \]
    thus $T^*(n)\sim \frac{1}{\ln n}$ when $n$ tends to $\infty$.
    If $n=2$, the map $\vec u_T$ has a unique fixed point 
    $\vec x=\left(\frac{1}{2}\,\,\frac{1}{2}\right)$
    if and only if $T\ge T^*(2)=\frac{1}{2}$.
\end{proposition}
\begin{proof}
    From Lemma~\ref{lem:all-ones}, $\vec x\in\mathrm{int}(\Sigma)$ is a fixed point of $\vec u_T$, 
    with $\vec x\neq\left(\frac{1}{n}\,\cdots\,\frac{1}{n}\right)$,
    if and only if there exists $\mathcal{K}\subset\set{1,\dots,n}$, $y,z\in\R$,
    such that $\vec x_i=y$ for $i\in\mathcal{K}$, $\vec x_i=z$ for $i\notin\mathcal{K}$,
    $0<k=\abs{\mathcal{K}}<n$, $ky+(n-k)z=1$, $y\e{-y/T}=z\e{-z/T}$, and $y<z$.
    Denote $\alpha=\frac{1}{ny}$.  Since $y<\frac{1}{n}$, we get necessary that $\alpha>1$.
    From $y\e{-y/T}=z\e{-z/T}$, we get $T=T_{\alpha,k}$, where
    \[
      T_{\alpha,k} = \frac{1-\frac{1}{\alpha}}{(n-k)\ln\left(\frac{(\alpha-1)n}{n-k}+1\right)}\enspace.
    \]
    This implies that $\vec u_T$ has a fixed point $\vec x\in\mathrm{int}(\Sigma)$, 
    $\vec x\neq\left(\frac{1}{n}\,\cdots\,\frac{1}{n}\right)$
    if and only if $T\in\mathcal{T}=\set{T_{\alpha,k}, \alpha>1,k\in\set{1,\dots,n-1}}$.
    Let
    \[
      T^*(n) = \sup_{\alpha>1} T_{\alpha,n-1}
      = \sup_{\alpha>1}\frac{1-\frac{1}{\alpha}}{\ln((\alpha-1)n+1)}\enspace.
    \]
    We shall show that $\mathcal{T}={]0,T^*(n)]}$ when $n>2$ and $\mathcal{T}={]0,T^*(2)[}$ when $n=2$.

    First, a study of $T_{\alpha,k}$ as a function of $k$ shows that it is increasing.  It is therefore sufficient to show that $\set{T_{\alpha,n-1}, \alpha>1}={]0,T^*(n)]}$ when $n>2$, and $\set{T_{\alpha,1}, \alpha>1}={]0,T^*(2)[}$ when $n=2$.
    Second, a study of $T_{\alpha,n-1}$ as a function of $\alpha>1$ shows that, when $n>2$, $T_{\alpha,n-1}$ is increasing, then decreasing, tends to $0$ when $\alpha$ goes to infinity, and its maximum is attained for $\alpha=\alpha_n$, where $\alpha_n>\frac{2(n-1)}{n}$.  Hence $\mathcal{T}={]0,T^*(n)]}$.  When $n=2$, $T_{\alpha,1}$ is decreasing, tends to $0$ when $\alpha$ goes to infinity, and to $\frac{1}{2}$ when $\alpha$ goes to $1$.  Hence $T^*(2)=\frac{1}{2}$, and $\mathcal{T}={]0,T^*(2)[}$.

    Moreover, for $n\ge3$,
    \[
      T^*(n)=T_{\alpha_n,n-1}=\frac{1-\frac{1}{\alpha_n}}{\ln((\alpha_n-1)n+1)}
      <\frac{1}{\ln((\alpha_n-1)n+1)}\enspace,
    \]
    and since $\alpha_n>\frac{2(n-1)}{n}$, we get $T^*(n)<\frac{1}{\ln(n-1)}$.
    For the lower bound, we get $T_{\ln n,n-1}\le T^*(n)$, since $\ln n>1$.
\end{proof}

Proposition~\ref{prop:allones-bornes} deals with the very special case of a complete graph.
In more general circumstances, the exact computation of the critical temperature seems out of range. However, we can obain numerically a lower bound of the critical temperature, which seems to be an accurate estimate, using the following homotopy-type method.
We first choose two random initial vectors on the simplex. Then, we iterate the map $\vec f_T$ from each of these vectors. For small values of $T$, this yields with an overwhelming probability two different webranks. Then, we increase the temperature $T$, and keep iterating the map $\vec f_T$ on each of these webranks, until the two webranks coincide. This yields a lower bound of the critical temperature. Then, we repeat this procedure, with new random initial vectors, until the lower bound of the critical temperature is not improved any more.
Note that the simpler method consisting in keeping $T$ fixed and iterating $\vec f_T$ from various initial conditions (random vectors or Dirac distributions on a vertex of the simplex) experimentally yields an under estimate of the critical temperature.

Using the previously described homotopy-type method, we computed numerically the critical temperature for two families of graphs. These experiments reveal that the $1/\ln(n)$ asymptotics obtained for the complete graph gives a good general estimate. We first considered the ring graph, with $n$ nodes, in which node $i$ is connected to its two neighbors and to itself.  The critical temperature, for $n=51, 201, 501$ and $1001$ is shown by stars in Figure~\ref{fig:universality-Tcrit-all-ones}. The exact value of the critical temperature of the complete graph with $n$ nodes, $T^*(n)$, is drawn as a continuous curve. We see that the critical temperatures of the ring and complete graphs are essentially proportional. We also computed numerically the critical temperature for a standard model of random directed graph, in which the presence of the different arcs are independent random variables, and for every $(i,j)$, the probability of presence of the arc $(i,j)$ is given by the same number $p$. We took
$p=10/n$, so that every node is connected to an average number of $10$ nodes. The corresponding critical temperatures are represented by circles. The values of these critical temperatures do not seem to change significantly with the realization of the random graph, hence, each of the values which are represented correspond to a unique realization.
\begin{figure}[t]
\begin{center}
   \includegraphics[width=\textwidth]{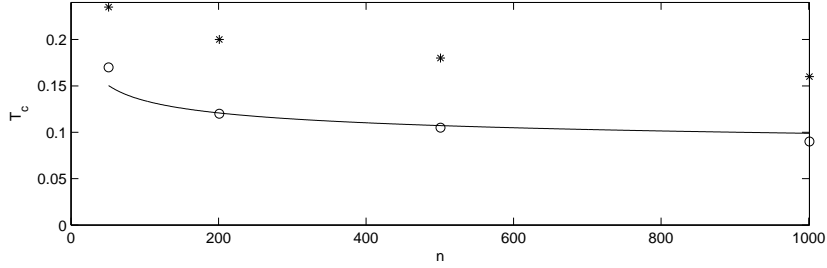}
    \caption{Estimation of the critical temperature $T_c$
    as a function of the number of nodes $n$
    for the complete graph (continuous curve),
    the ring graph (stars) and random graphs (circles).}
    \label{fig:universality-Tcrit-all-ones}
\end{center}
\end{figure}


\begin{remark}\label{rem:matrices22}
Let us briefly discuss about the case of an arbitrary $2\times2$ irreducible matrix $C$ with the weight function $g_T=\e{x/T}$. In this case, some elementary calculations give information about the critical temperature~\cite{AGN06}.
Firstly, the critical temperature for a graph of only two nodes is always less than $1$, since it can be shown that $\vec u_T$ has a unique fixed point if $T\ge1$. This is the best general upper bound that can be given for problems of this dimension since for every $T<1$, we can construct a $2\times2$ matrix such that $\vec u_T$ has at least two fixed points.
Moreover, one can show that, for every $T>0$, the map $\vec u_T$ has at most $5$ fixed points and does not have any orbit of period greater than $1$.
Numerical experiments show that for a $2\times2$ irreducible matrix, the number of fixed points of the map $\vec u_T$ can change $0,1,2$ or even $3$ times when decreasing the temperature $T$.

This can be seen on Figures~\ref{fig:135ptsfxes}, which were obtained experimentally. Let $\alpha=C_{11}/C_{12}$ and $\beta=C_{22}/C_{21}$. For a specified $\beta$, $\vec u_T$ has $5$ fixed point if $(T,\alpha)$ belongs to the black region, $3$ fixed points in the grey region and $1$ fixed point in the white region.
  \begin{figure}[t]
  \begin{center}
    \includegraphics[width=\textwidth]{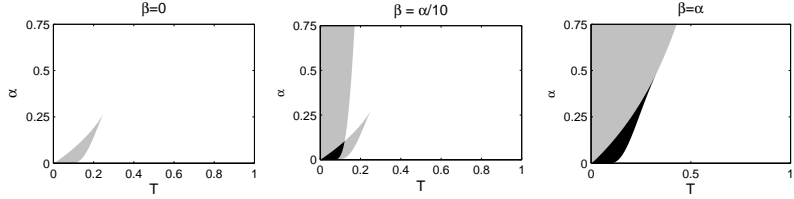}
    \caption{For $2\times2$ matrices, the map $\vec u_T$ has $5$ fixed points
    for $(T,\alpha)$ in the black region,
    $3$ fixed points for $(T,\alpha)$ in the grey region,
    and $1$ fixed point otherwise.}
    \label{fig:135ptsfxes}
  \end{center}
  \end{figure}
\end{remark}


\section{Experiments on a subgraph of the Web}\label{sec:application-web}

In this section, we briefly present our experiments of the $T$-PageRank on a large-scale example.
We consider a subgraph of the Web with about $280000$ nodes which has been obtained by S. Kamvar from a crawl on the Stanford web\footnote{The adjacency matrix can be found on \texttt{http://www.stanford.edu/\~{}sdkamvar/research.html}.}.
We use the variant of our model presented in Section~\ref{sec:refinement}, with a transition matrix given by
\[ M(\vec x)_{ij} 
  = \gamma\,\frac{C_{ij}\,\e{\vec x_j/T}}{\sum_kC_{ik}\,\e{\vec x_k/T}}
  +(1-\gamma)\,\frac{\e{\vec x_j/T}}{\sum_k\e{\vec x_k/T}}\enspace,
\]
where we suppose that for each dangling node $i$ (i.e.\ a node corresponding to a webpage without hyperlink), the $i$th row of the matrix $C$ is a row of all ones.
The chosen damping factor is $\gamma=0.85$.
We have computed the $T$-PageRank from the recurrence~$(\ref{eq:iter-fT})$ for various temperatures $T$ and initial rankings.  As expected, when the temperature $T$ is large,  the $T$-PageRank is very close to the classical PageRank, and when $T$ approaches zero, arbitrary close initial rankings can induce totally different $T$-PageRanks.  The critical temperature experimentally seems to be about $T=0.033$. It has the same order of magnitude as the $T^*(n)=0.06148$ estimate discussed in Section~\ref{sec:particular-cases}.

As in~\cite{VLD07}, we represented in a log-log scale the cumulative distribution function of the PageRank, i.e\, the proportion of pages for which the $T$-PageRank is larger than a given value, as a function of this value. In Figure~(a), we show the successive $T$-PageRanks obtained for \emph{increasing} temperatures from $T=0.015$ to the critical temperature $T=0.033$ by the following variant of the previously described homotopy method: for $T=0.015$, we iterate the map $\vec f_T$, with a Dirac mass on a vertex of the simplex as initial ranking, until a fixed point is reached. Then, for each new value of $T$, we iterate $\vec f_T$ until a fixed point is reached, starting from the previous fixed point. For $T\le0.032$ the distribution of the $T$-PageRank is quite different from this of the PageRank and it comes closer suddenly for $T=0.033$.
In Figure~(b), we show the successive $T$-PageRanks obtained by a similar method for \emph{decreasing} temperatures from $T=0.033$ to $0.009$, with the classical PageRank as an initial ranking. The latter procedure may be compared with simulated annealing schemes, in which the temperature is gradually decreased. Until $T=0.0091$, the distribution of the $T$-PageRank is quite similar to that of the classical PageRank (see a zoom in Figure~(c)). With $T=0.009$, the $T$-PageRank moves suddenly away from the PageRank. These figures suggest that the gap between webpages considered as ``good'' and ``bad'' is more pronounced with the $T$-PageRank than with the classical PageRank.
\begin{figure}[htb]
\centering
 \subfloat[][]{\includegraphics[width=0.9\textwidth]{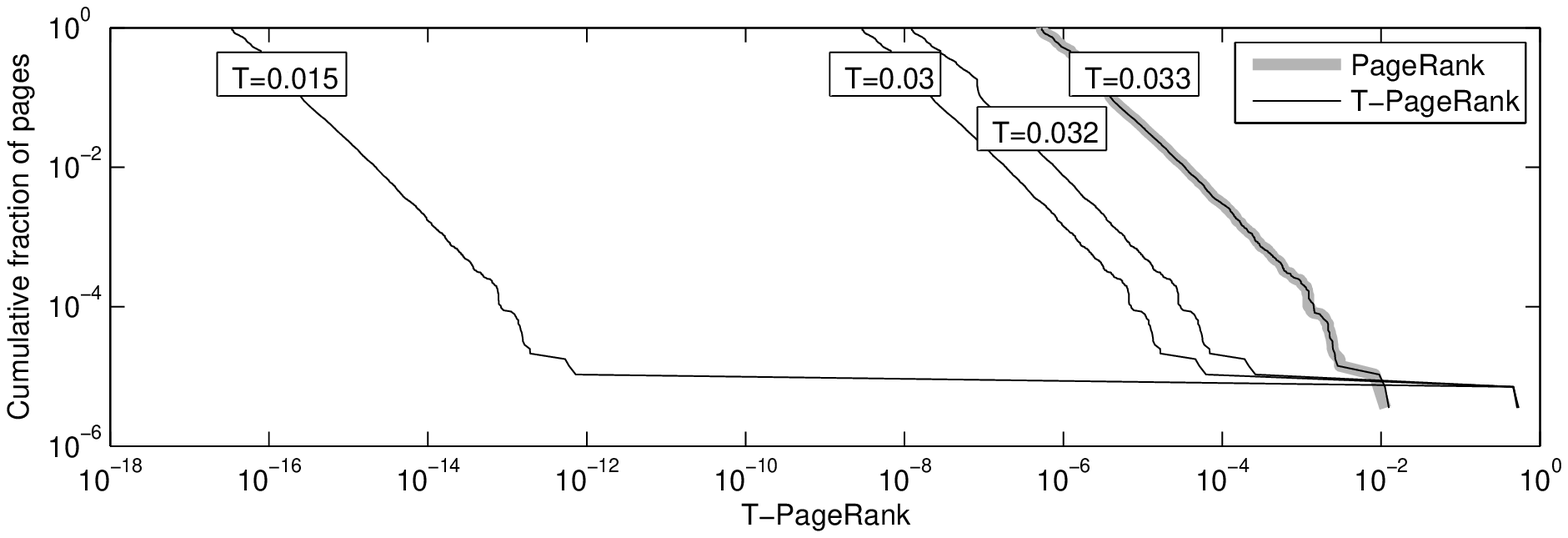}}
  \subfloat[][]{\includegraphics[width=0.9\textwidth]{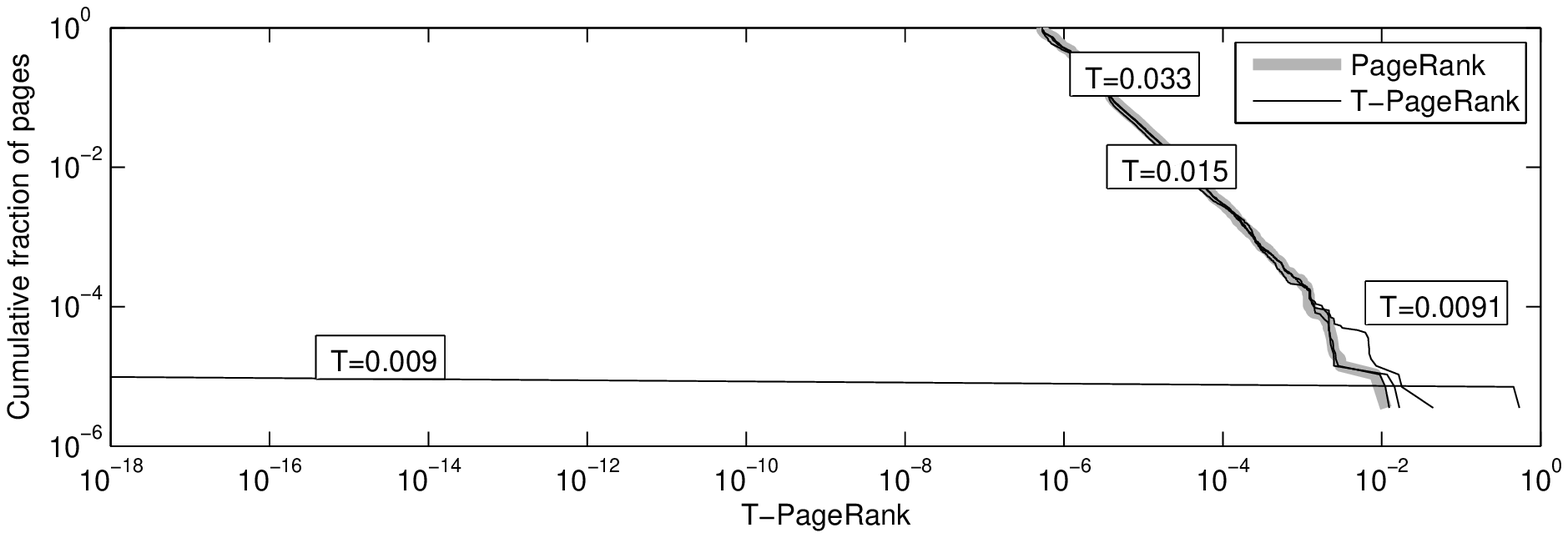}}
  \subfloat[][]{\includegraphics[width=0.9\textwidth]{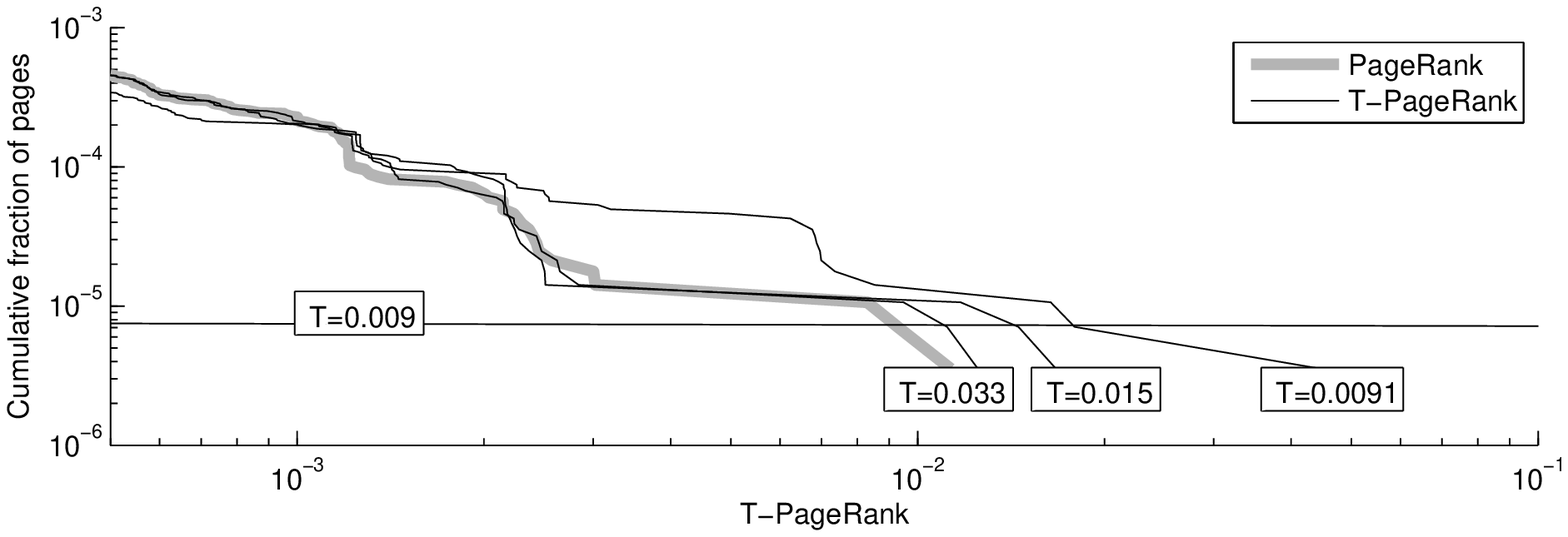}}
    \caption{
       Fraction of pages having a PageRank larger than a particular value.
       (a) $T$-PageRanks computed with increasing temperatures from $T=0.015$ to $0.033$,
       with a vertex of the simplex as initial ranking.
       (b) $T$-PageRanks computed with decreasing temperatures from $T=0.033$ to $0.009$,
       with the classical PageRank vector as initial ranking.
       (c) Zoom of Figure~(b).}
    \label{fig:stanfordscores-distribution}
\end{figure}

We have also compared the five best nodes for the classical PageRank and for the $T$-PageRank with decreasing temperatures $T=0.033$, $0.015$ and $0.0091$. As we see in Table~\ref{tab:top-5}, for $T=0.033$, the PageRank and $T$-PageRank give a similar ranking for the top-five. But for smaller temperatures as $T=0.015$ or $T=0.0091$, even the two best nodes are exchanged.
\begin{table}[htb]
\centering
\begin{tabular}{ccc}
$T=0.033$&$T=0.015$&$T=0.0091$\\
\hline
      1 &     2 &     2 \\
      2 &     1 &     1 \\
      3 &     3 &     3 \\
      4 &     6 &     46 \\
      5 &     7 &     33
\end{tabular}
\caption{The five best nodes of the $T$-PageRank for several values of~$T$: the numbers refer to the rankings acccording to the classical PageRank.}\label{tab:top-5}
\end{table}

Since for this special set of data, the correspondence between the page numbers and the urls is not available, one cannot interpret the discrepancies between the PageRank and the $T$-Pagerank in Table 6.1. In~\cite{Pov07}, J.-P. Poveda made similar experiments on the larger matrix obtained by S. Kamvar for a crawl of the union of the Stanford and Berkeley Webs\footnote{Also available on \texttt{http://www.stanford.edu/\~{}sdkamvar/research.html}.}, with about 685000 nodes, for which, this time, the correspondence between some pages and the main urls is given. These experiments suggest that the $T$-PageRank obtained by the latter scheme, in which the temperature is gradually decreased, as illustrated in Figure~(b), might be of practical interest. A full experimentation, on a real scale web, is beyond the scope of the present paper.

As a concluding remark, we would like to point out that our results might be considered as an argument in favor of the thesis that one should not use PageRank type measures to assess quality. Indeed, the validity of the classical PageRank relies on an ideal view of the web, in which the makers of pages are thought of as experts, creating hyperlinks only to pages they carefully examined, and judged by themselves to be of interest. In the real world, however, the web makers may be influenced by factors like reputation, to which the webrank participates. Within the limits of the model, our results show that pathological phenomena, like getting non unique and even not meaningful rankings, may occur in a self-referential world in which the websurfers would excessively rely on the web ranking, rather than on their own judgement.

\end{document}